\documentclass{article}

\usepackage[english]{babel}
\usepackage{amsthm}
\usepackage{amssymb}
\usepackage[
backend=biber,
babel=other,
maxbibnames=15,
sorting=none
]{biblatex}
\usepackage{algorithm}
\usepackage{algorithmic}
\newcommand\br[1]{\left(#1\right)}
\newcommand\fbr[1]{\left\{#1\right\}}
\renewcommand{\leq}{\leqslant}
\renewcommand{\geq}{\geqslant}

\usepackage[letterpaper,top=2cm,bottom=2cm,left=3cm,right=3cm,marginparwidth=1.75cm]{geometry}

\usepackage{amsmath}
\usepackage{graphicx}
\usepackage[colorlinks=true, allcolors=blue]{hyperref}
\newtheorem{lemma}{Lemma}
\newtheorem{theorem}{Theorem}
\newtheorem{definition}{Definition}
\newtheorem{remark}{Remark}

\title{Adaptive Variant of Frank-Wolfe Method for Relative Smooth Convex Optimization Problems}
\author{A. A. Vyguzov, F. S. Stonyakin}
\date{}

\begin{document}
\maketitle

\begin{abstract}
This paper proposes a new variant of the adaptive Frank-Wolfe algorithm for relatively smooth convex minimization problems. It suggests using a divergence different from half of the squared Euclidean norm in the step size adjustment formula. Convergence rate estimates for this method are proven for minimization problems involving relatively smooth convex functions with the triangle scaling property. We also conducted computational experiments for the Poisson linear inverse problem and SVM models. The paper also identified the conditions under which the proposed algorithm shows a clear advantage over the adaptive proximal gradient Bregman method and its accelerated variants.
\end{abstract}

\section{Introduction}

Relatively smooth functions were introduced a few years ago in \cite{rel_smooth, lu2018relatively} to generalize the class of L-smooth problems, for which complexity estimates of gradient-type methods can be proven in high-dimensional spaces. The importance of studying such functions is explained by the significant extension of the class of problems for which we can apply widely known first-order minimization methods: distributed optimization problems with centralized and decentralized networks (see \cite{Fioretto_2018}), a lot of machine learning algorithms such as support vector machine (SVM) method for binary classification with a non-smooth loss function, Poisson linear regression for signal recognition in noise, etc.

This is what already has been studied for this class of problems
\begin{itemize}
	\item convergence rate results for the non-accelerated proximal method are known to be $O(1/k)$ (a general overview can be found in \cite{rel_smooth_review}),
	\item in \cite{rel_smooth_acc}, $O(1/k^2)$ estimates were obtained, but only for a specific subclass of problems with the triangle scaling condition (TSE), see \eqref{triangle_scaling_property} for details.
\end{itemize}
In our work, we aim
\begin{itemize}
	\item to investigate the Frank-Wolfe algorithm for relatively smooth functions,
	\item to prove its linear convergence rate,
	\item to demonstrate that in practice it can perform more efficiently than its analogs from \cite{rel_smooth_acc}.
\end{itemize}
The study of the Frank-Wolfe method is also of great interest because it is much simpler to implement than accelerated gradient methods, requiring less computational power, which is, of course, a critical requirement in real-world practice.

Let us introduce some notations that will be used frequently in this work:
\begin{itemize}
	\item $V(x, y)$ - Bregman divergence at points $x$ and $y$: $V(x, y) = h(x) - h(y) - \langle \nabla h(y), x - y \rangle$ where $h$ is a convex function (not necessarily strongly convex).
	\item $x^*$, $f^*$ - a minimum point and the value of the function $f$ at $x^*$.
	\item $d_k$ - descent direction $d_k = s_k - x_k$ of the Frank-Wolfe method, where $s_k \in LMO_Q(\nabla f(x_k))$ and $LMO_Q(g) = \text{argmin}_{z \in Q} \ g^\top z$  (LMO - linear minimization oracle),  $Q$ is a convex and compact subset of $\mathbb{R}^n$.
\end{itemize}

We propose an adaptive Frank-Wolfe algorithm for the minimization of relatively smooth functions. To do this, let us recall the definition:

\begin{definition}
Function $f$ is called a relatively smooth, if it is differentiable on $Q$ and the following inequality holds
\begin{equation}\label{rel_smooth}
    f (x) \leq f (y) + \langle \nabla f (y), x - y\rangle + LV (x, y)
\end{equation}
for all $x, y \in Q$.
\end{definition}

Relatively smooth functions were first introduced in \cite{rel_smooth}.

We emphasize that relative smoothness is a generalization of ordinary smoothness with the following reference function: $h(x) = \frac{1}{2} \|x\|_2^2$ (Euclidean distance). For instance, the Poisson inverse problem, D-optimal design, the quadratic function $x^4$, and many other functions are not smooth functions over the entire space in the classical sense, but they are relatively smooth with respect to a certain reference function.

Similarly, the empirical risk minimization (ERM) problem (see, for example, \cite{distr_min}) is smooth in the usual sense. However, if there is a centralized architecture and similarity of functions, and the central node only knows 'its own' function, it is more advantageous to switch to a non-Euclidean divergence and consider their relative analogs instead of smoothness and strong convexity. In this way, we improve the conditioning of the function.

Now that we have established the usefulness of this class of functions, let us recall the concept of the triangle scaling property from \cite{rel_smooth_acc} (Definition 2). It is one of the key properties in proving the final convergence rate of our algorithm.

\begin{definition}
	Let $f$ be a convex relatively smooth function. We say that the Bregman divergence $V$ possesses the triangle scaling property if there exists constant $\gamma > 0$ such that for all $x, z, \Tilde{z} \in \text{rint} \ \text{dom} \ h$,
	\begin{equation}\label{triangle_scaling_property}
		V((1 - \theta)x + \theta z, (1 - \theta)x + \theta \Tilde{z}) \leq \theta^\gamma V(z, \Tilde{z}), \ \forall \theta \in [0, 1],
	\end{equation}
	where $\gamma$ is called the uniform triangle scaling exponent (TSE).
\end{definition}

\section{Algorithm}
Further, we follow the approach presented in \cite{aivazian2023adaptive}, where an adaptive variant of the Frank-Wolf algorithm for smooth Lipschitz continuous functions was proposed, and introduce its analog for a class of relatively smooth convex functions.

For $1 < \gamma \leq 2$ and $x_{k+1} = x_k + \alpha_k d_k$, define the following step size:
\begin{equation}\label{step_length1}
	\alpha_k := \min \fbr{ \left( \frac{- \langle \nabla f(x_k), d_k \rangle }{2 L_k V(x_k + d_k, x_k)} \right) ^{\frac{1}{\gamma - 1}}, 1 },
\end{equation}

Now let us introduce Algorithm \ref{alg:adapt_fw_rel} with step \eqref{step_length1}.

\begin{algorithm}[!ht]
\caption{Adaptive Frank-Wolfe algorithm with the Bregman divergence.}\label{alg:adapt_fw_rel}
\begin{algorithmic}[1]
   \REQUIRE $N$, $x_{0} \in Q, L_{-1}>0$.
   \FOR{$k=0, 1, \ldots, N-1$}
    \STATE $L_{k}=\frac{L_{k-1}}{2}$.
    \STATE $s_k = \text{argmin}_{x \in Q}\left\{ \nabla f(x_k)^\top x \right\}$.
    \STATE $d_k = s_k - x_k$.
    \STATE $\alpha_k = \min \left \{ \left( \frac{-\nabla f(x_k)^\top d_k}{2 L_k V(s_k, x_k)}\right)^{1 / (\gamma - 1)}, 1 \right \} $.
    \IF{$f(x_k + \alpha_k d_k) \leq f(x_k) + \alpha_k \nabla f(x_k)^\top d_k + \alpha_k^\gamma L_k V(s_k, x_k)$}
    \STATE $x_{k+1} = x_k + \alpha_k d_k$.
    \ELSE
    \STATE $L_{k}=2 L_{k}\text{ go to line ~5}.$
    \ENDIF
 \ENDFOR
\ENSURE $x_N$.
\end{algorithmic}
\end{algorithm}

Let us prove two auxiliary lemmas.

\begin{lemma}
	Let $f$ satisfy \eqref{rel_smooth}, let $V$ satisfy \eqref{triangle_scaling_property} with $1 < \gamma \leq 2$; then
	\begin{equation}\label{descent_tse}
		f(x_k + \alpha_k d_k) \leq f(x_k) +  \alpha_k \nabla f(x_k)^\top d_k + \alpha_k^\gamma L_k V(x_k + d_k, x_k).
	\end{equation}
\end{lemma}

\begin{proof}
	Using \eqref{triangle_scaling_property} we obtain
	\begin{gather*}
		V(x_k + \alpha_k d_k, x_k) = V((1 - \alpha_k)x_k + \alpha_k (x_k + d_k), (1 - \alpha_k)x_k + \alpha_k x_k) \leq\\
		\leq \alpha_k^\gamma V(x_k + d_k, x_k) = \alpha_k^\gamma V(s_k, x_k).
	\end{gather*}
	This completes the proof of the lemma.
\end{proof}

\begin{lemma}
	Let the step size satisfy \eqref{step_length1}, let Bregman divergence satisfy \eqref{triangle_scaling_property} with $1 < \gamma \leq 2$, then we have:
	\begin{equation}\label{descent_alpha_eq_1}
		f(x_{k+1}) - f^* \leq  \frac{f(x_k) - f^*}{2},
	\end{equation}
	if $\alpha_k = 1$, and
	\begin{equation}\label{descent_alpha_less_1}
		f(x_{k+1}) - f(x_k) \leq  - \frac{1}{2} \frac{(-\langle \nabla f(x_k), d_k \rangle)^{\gamma / (\gamma - 1)}} {(2L_k V(x_k + d_k, x_k))^{1 / (\gamma - 1)}},
	\end{equation}
	if $\alpha_k < 1.$
\end{lemma}

\begin{proof}
	If $\alpha_k = 1$, then, using \eqref{rel_smooth} and the identity $x_{k+1} = x_k + d_k$, we obtain
	\begin{gather*}
		f(x_{k+1}) \leq f(x_k) + \langle \nabla f(x_k), d_k \rangle + L_k V(x_{k+1}, x_k).
	\end{gather*}
	Therefore,
	\begin{gather*}
		f(x_{k+1}) - f^* \leq f(x_k) - f^* + \frac{1}{2} \langle \nabla f(x_k), d_k \rangle \leq \frac{f(x_k) - f^*}{2},
	\end{gather*}
	here we used the definition \eqref{step_length1}, which implies
	\[
	2L_k V(x_{k+1}, x_k) < - \langle \nabla f(x_k), d_k \rangle,
	\]
	as well as the convexity relation
	\[
	\langle \nabla f(x_k), d_k \rangle \leq \langle \nabla f(x_k), x^* - x_k \rangle \leq f^* - f(x_k).
	\]
	
	If $\alpha_k < 1$, then by \eqref{descent_tse} we have
	\begin{equation*}
		\begin{aligned}
			& f(x_k + \alpha_k d_k) - f(x_k) \leq \\ 
			& \leq - \frac{(- \langle \nabla f(x_k), d_k \rangle)^{\gamma / (\gamma - 1)} }{(2 L_k V(x_k + d_k, x_k))^{1 / (\gamma - 1)}} + \frac{(- \langle \nabla f(x_k), d_k \rangle)^{\gamma / (\gamma - 1)}}{2^{\gamma / (\gamma - 1)} (L_k V(x_k + d_k, x_k))^{1 / (\gamma - 1)}} = \\ 
			& = - \frac{1}{2} \frac{(-\langle \nabla f(x_k), d_k \rangle)^{\gamma / (\gamma - 1)}}{(2L_k V(x_k + d_k, x_k))^{1 / (\gamma - 1)}}.
		\end{aligned}
	\end{equation*}
\end{proof}

We are ready to proof the main estimation of Algorithm \ref{alg:adapt_fw_rel}.

\begin{theorem}
	Let $V(x,y) \leq \frac{R^2}{2}$ for all $x,y \in Q$ and $R > 0$; then with $1 < \gamma \leq 2$ for the Algorithm \ref{alg:adapt_fw_rel} the following inequality holds
	\begin{equation}\label{main_estimation}
		f(x_k) - f^* \leq  \left( \frac{2}{k+2} \right)^{\gamma - 1} \max_{j \in \overline{0,k-1}} L_j R^2 \hspace{1cm} \forall k \geq 1.
	\end{equation}
\end{theorem}

\begin{proof}
	The proof is done by induction
	
	1) Basis $k=1$.
	
	Using $\alpha_0 = 1$, we obtain:
	\begin{equation*}
		\begin{aligned}
			& f(x_1) - f^* \leq \\
			& \leq f(x_0) - f^* + \langle \nabla f(x_0), x_1 - x_0 \rangle + L_0 V(x_1, x_0) \leq \\
			& \leq f(x_0) - f^* + \langle \nabla f(x_0), x^* - x_0 \rangle + L_0 V(x_1, x_0) \leq \\
			& \leq L_0 V(x_1, x_0) \leq \frac{L_0 R^2}{2} \leq \frac{2L_0 R^2}{3} \leq \br{\frac{2}{1 + 2}}^{\gamma - 1} L_0 R^2.
		\end{aligned}
	\end{equation*}
	
	The third inequality follows by $f(x_0) - f^* + \langle \nabla f(x_0), x^* - x_0\rangle \leq 0$, and the last by $0 < \gamma - 1 \leq 1$.
	
	Using $\alpha_0 < 1$, we obtain
	
	\begin{equation*}
		\begin{aligned}
			& f(x_0) - f^* \leq \langle \nabla f(x_0), x_0 - x^* \rangle = - \langle \nabla f(x_0), x^* - x_0 \rangle \leq - \langle \nabla f(x_0), d_0 \rangle \leq \\ 
			& \leq 2 L_0  V(x_0 + d_0, x_0) \leq L_0 R^2  = \br{\frac{2}{0 + 2}}^{\gamma - 1} L_0 R^2.
		\end{aligned}
	\end{equation*}
	
	Next, by \eqref{descent_alpha_less_1} we have
	
	\begin{equation*}
		\begin{aligned}
			f(x_1) - f^* \leq (f(x_0) - f^*) \left( 1 - \frac{1}{2} \left( \frac{f(x_0) - f^*}{2 L_0 V(x_0 + d_0, x_0)} \right)^{1 / (\gamma - 1)} \right),
		\end{aligned}
	\end{equation*}
	
	The last equation implies that
	\begin{equation*}
		\begin{aligned}
			f(x_1) - f^* \leq (f(x_0) - f^*) \left( 1 - \left( \frac{f(x_0) - f^*}{2^{\gamma - 1} L_0 R^2} \right) ^{1 / (\gamma - 1)} \right).
		\end{aligned}
	\end{equation*}
	
	So if $\frac{f(x_0) - f^*}{2^{\gamma - 1} L_0 R^2} > \br{\frac{1}{3}}^{\gamma - 1}$, then
	\begin{equation*}
		\begin{aligned}
			& 1 - \br{\frac{f(x_0) - f^*}{2^{\gamma - 1} L_0 R^2}}^{\frac{1}{\gamma - 1}} < 1 - \frac{1}{3} = \frac{2}{3} \\
			\text{and} \\
			& f(x_1) - f^* \leq \frac{2}{3} (f(x_0) - f^*) \leq \frac{2}{3} L_0 R^2 \leq \left(\frac{2}{3}\right)^{\gamma - 1} L_0 R^2 \ \ \text{with} \ 0 < \gamma - 1 \leq 1.
		\end{aligned}
	\end{equation*}
	
	If $\frac{f(x_0) - f^*}{2^{\gamma - 1} L_0 R^2} \leq \br{\frac{1}{3}}^{\gamma - 1}$, then
	\begin{equation*}
		\begin{aligned}
			f(x_1) - f^* \leq f(x_0) - f^* \leq \br{\frac{2}{3}}^{\gamma - 1} L_0 R^2.
		\end{aligned}
	\end{equation*}
	
	The basis is proven.
	
	2) Inductive step.
	
	Suppose that $f(x_k) - f^* \leq \br{\frac{2}{k+2}}^{\gamma - 1} L_{k-1} R^2$. Let us prove the equation:
	\begin{equation*}
		\begin{aligned}
			f(x_{k+1}) - f^* \leq \br{\frac{2}{k+3}}^{\gamma - 1} \max_{j \in \overline{0,k}} L_j R^2.
		\end{aligned}
	\end{equation*}
	
	Since $\alpha = 1$, using \eqref{descent_alpha_eq_1}
	\begin{equation*}
		\begin{aligned}
			& f(x_{k+1}) - f^* \leq \frac{f(x_k) - f^*}{2} < \left( \frac{k+2}{k+3} \right)^{\gamma - 1}(f(x_k) - f^*) \leq \\ 
			& \leq \left(\frac{2}{k+2} \right)^{\gamma - 1} \left(\frac{k + 2}{k+3} \right)^{\gamma - 1} L_{k-1} R^2 \leq \left(\frac{2}{k+3} \right)^{\gamma - 1} \max_{j \in \overline{0,k}} L_j R^2.
		\end{aligned}
	\end{equation*}
	
	Since $\alpha_k < 1$, using \eqref{descent_alpha_less_1} $f(x_{k+1}) - f^* \leq (f(x_k) - f^*) \left(1 - \left( \frac{f(x_k) - f^*}{2^{\gamma - 1} L_k R^2} \right)^{1 / (\gamma - 1)} \right)$. Let us consider two cases.
	
	In the case of $f(x_k) - f^* \leq \br{\frac{2}{k+3}}^{\gamma - 1} L_k R^2 \leq \br{\frac{2}{k+3}}^{\gamma - 1} \max_{j \in \overline{0,k}} L_j R^2$ the target inequality follows from the fact that $f(x_{k+1}) - f^* \leq f(x_k) - f^*$.
	
	In the remaining case $f(x_k) - f^* > \br{\frac{2}{k+3}}^{\gamma - 1} L_k R^2$ with $\frac{f(x_k) - f^*}{2^{\gamma - 1} L_k R^2} > \br{\frac{1}{k+3}}^{\gamma - 1}$ we obtain
	
	\begin{equation*}
		\begin{aligned}
			& f(x_{k+1}) - f^* \leq (f(x_k) - f^*) \left( 1 - \frac{f(x_k) - f^*}{2^{\gamma - 1} L_k R^2} \right) ^{1 / (\gamma - 1)} < \\
			& < (f(x_k) - f^*)\br{1 - \frac{1}{k+3}} = \frac{k + 2}{k + 3} (f(x_k) - f^*) \leq \\
			& \leq\br{\frac{k+2}{k+3}}^{\gamma - 1}\br{\frac{2}{k+2}}^{\gamma - 1} L_k R^2 \leq \br{\frac{2}{k+3}}^{\gamma - 1} \max_{j \in \overline{0,k}} L_j R^2,
		\end{aligned}
	\end{equation*}
	This proves the theorem.
\end{proof}

Additionally, we emphasize that the estimation \eqref{main_estimation} is optimal, since if we substitute $\gamma = 2$ and Euclidean reference function $\frac{1}{2} \|x\|_2^2$ in \eqref{main_estimation}, then we get the expression for the classic Frank-Wolfe algorithm that is optimal estimation.

From the Algorithm \ref{alg:adapt_fw_rel} listing, it can be seen that at each iteration we adjust $L_k$ variable. Let us prove that this does not downgrade the convergence rate. To do this, let us analyze the number of $L_k$ checks at each iteration.

\begin{remark}\label{remark_1}
If $f$ has $L$-Liptshitz continuous gradient and $L_{-1} \leq 2L$, then using \cite{fw_friends} (Lemma 1) we obtain $L_k \leq 2L $.
\end{remark}

\begin{remark}\label{remark_main_estimation_l_smooth}
	If $f$ has $L$-Liptshitz continuous gradient then the estimation \eqref{main_estimation} can be rewritten as $f(x_k) - f^* \leq  \left( \frac{2}{k+2} \right)^{\gamma - 1} L R^2$.
\end{remark}

\begin{remark}\label{remark_l_checks}
	Suppose that $f$ has $L$-Lipthitz continuous gradient, and at step $k$ $i_k$ inequality checks were performed from line~6 of the Algorithm~\ref{alg:adapt_fw_rel}; then, using $L_{-1} \leq 2L$ and \eqref{remark_1} for the Algorithm \ref{alg:adapt_fw_rel}, we have:
\begin{align*}
        &i_0 + i_1 +\ldots + i_N =\\
        &\quad= \left(2 + \log_{2}{\frac{L_1}{L_0}}\right)+ \ldots +\left(2 +\log_{2}{\frac{L_N}{L_{N-1}}}\right)=
        \\&\quad= 2N + \log_{2}{\left(\frac{L_1}{L_0}\frac{L_2}{L_1} \cdots \frac{L_N}{L_{N-1}}\right)}=
        \\&\quad= 2N + \log_{2}{\frac{L_N}{L_0}}   \leq 2N + \log_2{\frac{2L}{L_0}} = O(N).
\end{align*}
Thus, the total number of inequality checks from line 6 of the Algorithm \ref{alg:adapt_fw_rel} after $N$ iterations is $O(N)$.
\end{remark}

\subsection{Convergence of the proposed method at a geometric rate for relatively strongly convex problems.}

In this section, we will show improved convergence estimates for Algorithm \ref{alg:adapt_fw_rel}. Specifically, we will demonstrate linear convergence for the property of relative strong convexity, an analog of the classical strong convexity with the Euclidean norm. Why are the conditions of strong convexity and smoothness alone insufficient for linear convergence? In \cite{polyak_book}, p. 190, an example of a strongly convex function is given that converges slower than linearly, so additional conditions are necessary to guarantee linear convergence. For the standard Frank-Wolfe algorithm, such a condition can be, for instance, the angular property (see, for example, \cite{fw_method_common}). By analogy with \cite{fw_friends}, we will introduce this property using divergence, hence the subsequent reasoning differs from that proposed in \cite{fw_friends}.

\begin{definition}
	A function $f$ is called relatively strongly convex if it is differentiable on $Q$ and there exists $\mu > 0$ such that the following condition holds
	\begin{equation}\label{rel_strong_conv}
		f (x) \geq f (y) + \langle \nabla f (y), x - y\rangle + \mu V (x, y)
	\end{equation}
	for all $x, y \in Q$.
\end{definition}

Next, we will introduce the angular condition adapted to the divergence.
\begin{definition}
	The angle condition is defined as follows
	\begin{equation} \label{angle_cond_def}
		\begin{aligned}
			\frac{ - \nabla f(x_k)^\top d_k}{V(s_k, x_k)} \geq \frac{\tau}{V(x^*, x_k)} (-\nabla f(x_k))^\top(x^* - x_k)
		\end{aligned}
	\end{equation}
	for fixed $\tau > 0$ and $x^* \in \text{argmin}_{x \in Q} f(x)$.
\end{definition}

Now, let us formulate a lemma. Note that according to this lemma, we obtain a guaranteed reduction of the residual by a factor of 2 if the step is full on the iteration:

\begin{lemma} \label{strong_conv_estimation}
	Suppose $f$ satisfies inequalities \eqref{rel_strong_conv} and \eqref{angle_cond_def}, then for Algorithm \ref{alg:adapt_fw_rel} the following holds:
	\begin{equation*}
		\begin{aligned}
			& f\left(x_k\right)-f^* \leqslant\left(f\left(x_0\right)-f^*\right) \prod_{i=1}^k \varphi_i \\
			& \varphi_i= \begin{cases}\frac{1}{2}, & \text { if } \alpha_i=1 \\
				1 - \frac{\tau}{2} \frac{\gamma ^{\gamma / (\gamma + 1)} }{\gamma + 1} \left( \frac{\mu}{2 L_i} \right) ^{1 / (\gamma-1)}, & \text { if } \alpha_i<1\end{cases} \\
			&
		\end{aligned}
	\end{equation*}
	where $\tau > 0$.
\end{lemma}

\begin{proof}
	Using \eqref{rel_strong_conv} we get
	\begin{equation*}
		\begin{aligned}
			& f^* - f(x_k) \geq \nabla f(x_k)^\top(x^* - x_k) + \mu V(x^*, x_k) \geq \\ 
			& \geq \min_\alpha \{ \alpha \nabla f(x_l)^\top (x^* - x_k) + \alpha^\gamma \mu V(x^*, x_k)\}
		\end{aligned}
	\end{equation*}
	
	Let us find $\alpha$ such that
	\begin{equation} \label{strong_conv_with_alpha}
		\begin{aligned}
			\nabla f(x_k)^\top(x^* - x_k) + \gamma \alpha^{\gamma - 1} \mu V(x^*, x_k) = 0
		\end{aligned}
	\end{equation}
	Therefore, we get
	$$\alpha^{\gamma - 1} = \frac{-\nabla f(x_k)^\top (x^* - x_k)}{\gamma \mu V(x^*, x_k)}$$
	
	Substituting this in \eqref{strong_conv_with_alpha}, we obtain
	
	\begin{equation} \label{precision_strong_conv}
		\begin{aligned}
			& f^* - f(x_k) \geq \\ 
			& \geq \left( - \frac{\nabla f(x_k)^\top (x^* - x_k)}{\gamma \mu V(x^*, x_k)} \right)^{1 / (\gamma - 1)} \nabla f(x_k)^\top(x^* - x_k) + \\ 
			& + \left( - \frac{\nabla f(x_k)^\top (x^* - x_k)}{\gamma \mu V(x^*, x_k)} \right)^{\gamma / (\gamma - 1)} \mu V(x^*, x_k) \geq \\
			& \geq \frac{(\gamma + 1)}{\gamma^{\gamma / (\gamma - 1)}} \frac{(- \nabla f(x_k)^\top(x^* - x_k)^{\gamma / \gamma - 1}}{(\mu V(x^*, x_k))^{1 / \gamma - 1}}
		\end{aligned}
	\end{equation}
	
	Using $h_k = f(x_k) - f^*$, we get by \eqref{descent_alpha_less_1}
	\begin{equation*}
		\begin{aligned}
			& h_{k+1} \leq h_k - \frac{1}{2} \frac{(- \nabla f(x_k)^\top d_k)^{\gamma / (\gamma - 1)}}{(2 L_k V(s_k, x_k))^{1 / (\gamma - 1)}} \leq \\ 
			& \leq h_k - \left( \frac{\tau}{2} \frac{1}{2 L_k}\right)^{1 / \gamma - 1} \frac{(- \nabla f(x_k)^\top (x^* - x_k))^{\gamma / (\gamma - 1)}}{V(x^*, x_k)^{1 / (\gamma - 1)}} \leq \\
			& \leq h_k \left( 1 - \frac{\tau}{2} \left( \frac{\mu}{2L_k}\right)^{1/(\gamma - 1)} \frac{\gamma^{\gamma / (\gamma - 1)}}{(\gamma + 1)} \right)
		\end{aligned}
	\end{equation*}
	In the first inequality we used the step size definition \eqref{step_length1}, in the second we used \eqref{angle_cond_def}, and in the third \eqref{precision_strong_conv}.
	
	If the step on iteration $k$ is full: $\alpha_k = 1$, then from \eqref{descent_alpha_eq_1} it follows that $f(x_{k+1}) - f^* \leq  \frac{f(x_k) - f^*}{2}$.
\end{proof}

Next, we find $\tau$ in Lemma \ref{strong_conv_estimation}, namely we detail the result on the method convergence to an $\epsilon$-optimal solution. We formulate this result as a theorem.
\begin{theorem}
	Suppose $f$ satisfies all the assumptions of Lemma \ref{strong_conv_estimation}, and let $x^*$ be an interior point of $Q$ in the sense of Euclidean norm. Then for Algorithm \ref{alg:adapt_fw_rel}, either the estimate $V(x^*, x_k) \leq \epsilon$ or the inequality holds:
	\begin{equation}\label{strong_conv_estimation_accurate}
		\begin{aligned}
			& f(x_{k+1}) - f^* \leq \\
			& \leq \left(1 - \frac{\delta}{D} \frac{\epsilon}{D_V} \frac{1}{2} \left( \frac{\mu}{2 \max_{j \in \overline{0,k}} L_j} \right) ^{1 / (\gamma-1)} \frac{\gamma ^{\gamma / (\gamma + 1)} }{\gamma + 1} \right)^k (f(x_0) - f^*),
		\end{aligned}
	\end{equation}
	where $\delta = \text{dist}(x^*, \partial Q) = \inf_{y \in \partial Q} \| x^* - y \|_2$, $\epsilon$ - required accuracy, $D = \max_{x,y} \| x - y\|_2$, $D_v = \max_{x,y} V(x, y)$, where $x, y \in Q$
\end{theorem}

\begin{proof}
	Let us denote $g = -\nabla f(x_k), \ \hat{g} = \frac{g}{\| g \|_2}$. Taking into account $x^* \in \text{Int(Q)}$, we obtain: $x^* + \delta \hat{g} \in Q$. Then:
	\begin{equation}\label{g_norm_ineq}
		\begin{aligned}
			g^{\top} d_k \geq g^{\top}\left(\left(x^*+\delta \widehat{g}\right)-x_k\right)=\delta g^{\top} \widehat{g}+g^{\top}\left(x^*-x_k\right) \geq \delta\|g\|_2+f(x_k)-f^* \geq \delta\|g\|_2
		\end{aligned}
	\end{equation}
	where we used $x^* + \delta \hat{g} \in Q$ in the first inequality and convexity of $f$ in the second.
	
	Next, if $V(x^*, x_k) > \epsilon$, then
	\begin{equation*}
		\begin{aligned}
			& g^{\top} \frac{d_k}{V(s, x_k)} \geqslant g^{\top} \frac{d_k}{D_v} \geqslant \frac{\delta\|g\|}{D_v} \geqslant \\
			& \geqslant \frac{\delta}{D_v} \cdot g^{\top}\left(\frac{x^*-x_k}{\left\|x^*-x_k\right\|}\right) \cdot \frac{V\left(x^*, x_k\right)}{V\left(x^*, x_k\right)}= \\
			& =\frac{\delta}{D_v} \cdot \frac{V\left(x^*, x_k\right)}{\left\|x^*-x_k\right\|} \cdot g^{\top}\left(\frac{x^*-x_k}{V\left(x^*, x_k\right)}\right) \geqslant \\
			& \geqslant \frac{\delta}{D} \cdot \frac{V\left(x^*, x_k\right)}{D_v} \cdot g^{\top}\left(\frac{x^*-x_k}{V\left(x^*, x_k\right)}\right) \geqslant \\
			& \geqslant \frac{\delta}{D} \cdot \frac{\varepsilon}{D_v} \cdot g^{\top}\left(\frac{x^*-x_k}{V\left(x^*, x_k\right)}\right)
		\end{aligned}
	\end{equation*}
	where we used \eqref{g_norm_ineq} in the second inequality and $V(x^*, x_k) > \epsilon$ in the last.
	
	In the remaining case $V(x^*, x_k) \leq \epsilon$, so the desired accuracy can be considered achieved.
	
	As we can see, the obtained inequality is the angular condition \eqref{angle_cond_def} with $\tau = \frac{\delta}{D} \cdot \frac{\varepsilon}{D_v}$. Combining the fact $\frac{\delta}{D} = \frac{\text{dist}(x^*, \partial C)}{D} \leq \frac{1}{2}$ and same reasoning as in Lemma \ref{strong_conv_estimation}, we get the initial inequality.
	
\end{proof}

\section{Experiments}
To demonstrate the performance of the proposed Algorithm~\ref{alg:adapt_fw_rel}, we conducted a series of numerical experiments. All computations were implemented in Python~3.11.4.

\subsection{D-optimal experiment design}\label{d_opt_subsec}

Consider the D-optimal experiment design problem
\begin{equation}\label{d_opt_design_rs}
	\begin{aligned}
		\min_x \quad & f(x):=  \log\left(\det\left(\sum_{i=1}^n x_i V_i V_i^T\right)\right), \\
		\text{s.t.} \quad & \sum_{i=1}^n x_i = 1, \quad x_i\geq 0, \quad i=1,\ldots,n. 
	\end{aligned}
\end{equation}
Here, $V_i\in \mathbb{R}^m$ for $i=1,\ldots,n$. It was shown in~\cite{lu2018relatively} that the objective function in the D-optimal experiment design problem is relatively smooth with respect to the reference function $h(x) = - \sum_{j=1}^{n} \ln(x_j)$, with relative smoothness constant $L=1$.

As benchmark datasets, we use the \emph{Housing} dataset containing 506 samples and 13 features, as well as the \emph{Bodyfat} dataset of dimension $252 \times 14$.

We compare the performance of Algorithm~\ref{alg:adapt_fw_rel} with the classical Frank--Wolfe method equipped with the short-step strategy and the Euclidean reference function:
\[
\alpha_k =
\min\left\{
1,
\frac{- \nabla f(x_k)^\top (s_k-x_k)}
{L\|s_k-x_k\|^2}
\right\},
\]
where $s_k$ denotes the solution of the linear minimization oracle (LMO). This comparison illustrates the advantage of the proposed method when minimizing relatively smooth objective functions. In both algorithms, a line-search procedure is employed to adaptively estimate the parameter $L$.

The experimental results are presented in Figure~\ref{exp:d_opt_design_rs:housing} for the Housing dataset and in Figure~\ref{exp:d_opt_design_rs:bodyfat} for the Bodyfat dataset. One can observe that Algorithm~\ref{alg:adapt_fw_rel} significantly outperforms the classical Frank--Wolfe method. The values of the parameter $L$, shown on the right-hand side of the figures, indicate that the proposed method provides substantially more accurate local approximations of the objective function. This suggests that the proposed geometry is more suitable for this class of problems.

\begin{figure}[ht]
	\includegraphics[width=16cm]{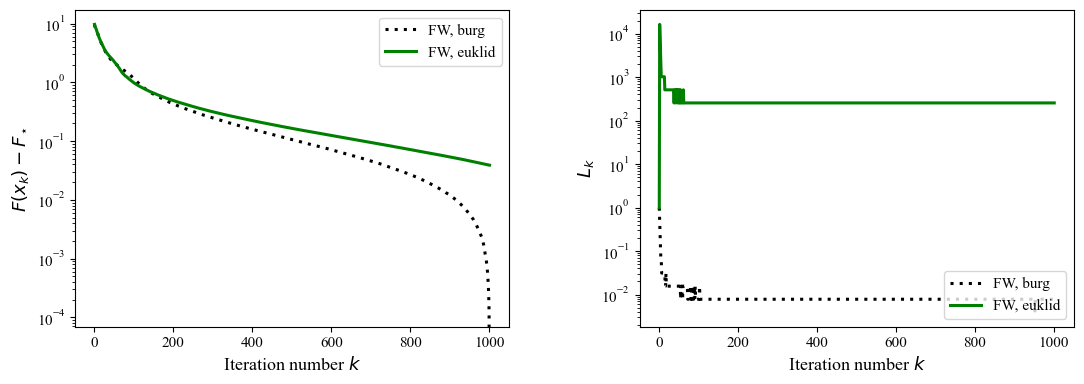}
	\centering
	\caption{Comparison of Algorithm~\ref{alg:adapt_fw_rel} (FW, Burg) and the classical Frank--Wolfe method with short-step strategy (FW, Euclid) for problem~\eqref{d_opt_design_rs} on the Housing dataset.}
	\label{exp:d_opt_design_rs:housing}
\end{figure}

\begin{figure}[ht]
	\includegraphics[width=16cm]{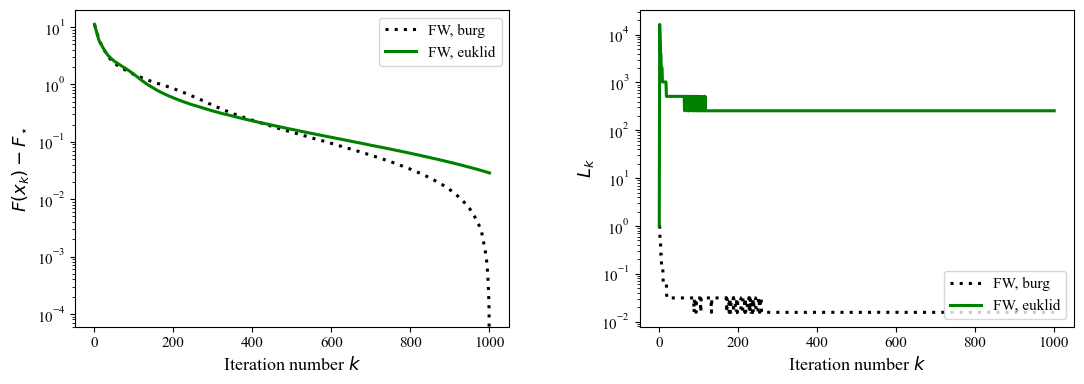}
	\centering
	\caption{Comparison of Algorithm~\ref{alg:adapt_fw_rel} (FW, Burg) and the classical Frank--Wolfe method with short-step strategy (FW, Euclid) for problem~\eqref{d_opt_design_rs} on the Bodyfat dataset.}
	\label{exp:d_opt_design_rs:bodyfat}
\end{figure}

\subsection{Poisson linear inverse problem}

Consider the well-known Poisson linear inverse problem. Let $A\in\mathbb{R}^{m\times n}$ be an observation matrix and let $b\in\mathbb{R}_{++}^m$ be a noisy observation vector. The goal is to recover a signal $x$ such that $Ax \approx b$. The optimization problem is formulated as
\begin{equation}\label{exp_pois_kl}
	\min_x f(x):= D_{\mathrm{KL}}(b, Ax),
\end{equation}
where
$D_{\mathrm{KL}}(x, y) = \sum_{i=1}^n \left(x^{(i)} \log\left(\frac{x^{(i)}}{y^{(i)}}\right) - x^{(i)} + y^{(i)} \right)$
denotes the Kullback--Leibler divergence. It was shown in~\cite{rel_smooth} that this objective function is relatively smooth with respect to the reference function $h(x) = - \sum_{i=1}^{n} \ln(x^{(i)})$, with relative smoothness constant $L=\|b\|_1$.

As the feasible set, we consider the intersection of the positive orthant with an $\ell_2$-ball. The noise level is set to $0.001$. The following problem dimensions were used:$n=200,\quad m=100$ and $n=500,\quad m=100$. The initial point was chosen near the center of the feasible region.

To improve the statistical reliability of the results, the residuals at each iteration were averaged over 20 independent runs.

\begin{figure}[ht]
	\includegraphics[width=16cm]{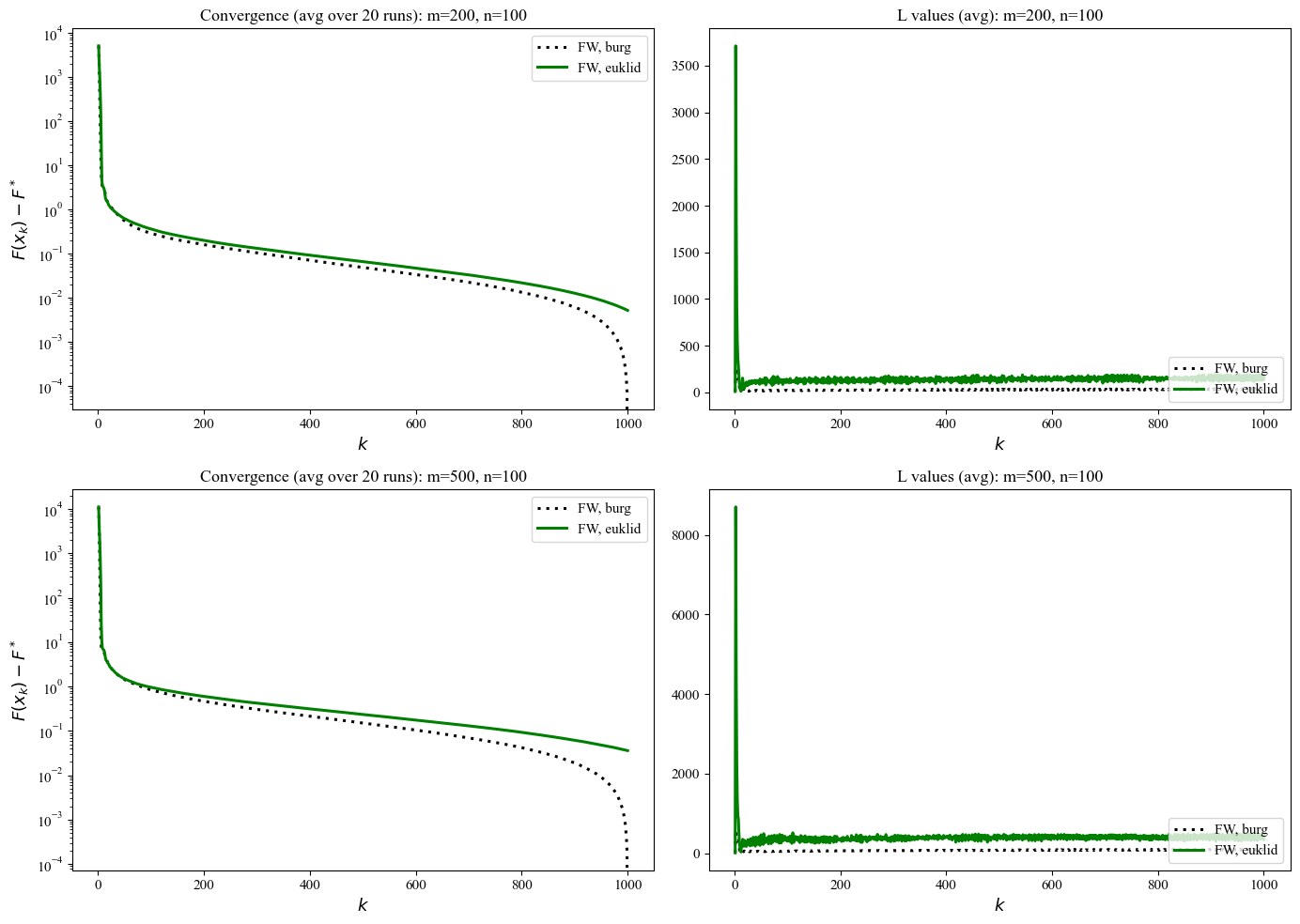}
	\centering
	\caption{Comparison of Algorithm~\ref{alg:adapt_fw_rel} (FW, Burg) and the classical Frank--Wolfe method with short-step strategy (FW, Euclid) for the Poisson linear inverse problem.}
	\label{kl_poisson}
\end{figure}

In this experiment, we use the same algorithms as in Section~\ref{d_opt_subsec}. Figure~\ref{kl_poisson} demonstrates that Algorithm~\ref{alg:adapt_fw_rel} consistently outperforms the classical Frank--Wolfe method. Moreover, the estimated values of the parameter $L$ remain significantly smaller throughout the optimization process, indicating that the Bregman-based approximation employed by the proposed method provides a more accurate local model of the objective function.

\section{Conclusion}
In our work, we proposed an adaptive Frank-Wolfe algorithm for relatively smooth problems \eqref{rel_smooth} and proved its convergence rate estimate. In numerous experiments, we demonstrated the conditions under which the proposed algorithm is applicable, using examples from the Poisson linear inverse problem and SVM tasks. The main advantage of this algorithm over accelerated and proximal algorithms is that it does not require the computation of projections or auxiliary subproblems such as argmin, which can be very computationally intensive for certain reference functions, especially in high dimensions. Despite this simplicity in implementation, it offers theoretical guarantees that are nearly as strong. Ultimately, we have a straightforward algorithm with a dimension-independent and norm-independent convergence rate.

\newpage
\printbibliography

\end{document}